\documentclass[letterpaper,11pt,reqno]{amsart}

\usepackage{amssymb}
\usepackage{amscd}
\usepackage{amsfonts}
\usepackage{amsmath}
\usepackage{amsthm}
\usepackage[all]{xy}
\usepackage{mathrsfs}
\usepackage[normalem]{ulem}
\usepackage{tikz}
\usepackage{etex}
\usepackage[paper = letterpaper, left = 30mm, right = 30mm, headsep = 7mm,
footskip = 10mm, top = 30mm, bottom = 30mm, footnotesep=5mm, headheight =
2cm]{geometry}

\usepackage{titletoc}
\titlecontents{part}[0pt]{}{\bfseries\chaptername\ \thecontentslabel\quad}{\bfseries}{\bfseries\hfill}
\usepackage{xr}

\usepackage[colorlinks=true]{hyperref}
\usepackage{url}
\usepackage{cite}
\usepackage{xcolor}
\renewcommand\thepart{\Roman{part}}

\def\id{\mathrm{id}}
\def\rank{\mathrm{rank}}

\def\part#1{%
  \vskip .02\vsize 
  \refstepcounter{part}
  \addcontentsline{toc}{part}{Part \thepart:\ #1}
  {\centering\large \textbf{Part \thepart}. #1\par}%
  \vskip .01\vsize
}

\newcommand{\Hom}{\operatorname{Hom}}

\newcommand{\vol}{\operatorname{vol}}
\newcommand*{\sheafHom}{\mathscr{H}\negthinspace om}

\theoremstyle{plain}
\newtheorem{theorem}{Theorem}[section]
\newtheorem*{theorem*}{Theorem}
\newtheorem{lemma}[theorem]{Lemma}

\newtheorem*{conjecture*}{Conjecture}

\newtheorem{corollary}[theorem]{Corollary}
\newtheorem*{corollary*}{Corollary}

\newtheorem{proposition}[theorem]{Proposition}
\theoremstyle{definition}
\newtheorem*{definition*}{Definition}

\numberwithin{equation}{section} 

\makeatletter
\def\ifdraft{\ifdim\overfullrule>\z@
  \expandafter\@firstoftwo\else\expandafter\@secondoftwo\fi}
\makeatother

\newsavebox{\ieeealgbox}

\newcommand{\Q}{\mathbb{Q}}

\newcommand{\Z}{\mathbb{Z}}

\def\hq{\hspace{-0.5mm}/\hspace{-0.14cm}/ \hspace{-0.5mm}}

\theoremstyle{definition}
\newtheorem{definition}[theorem]{Definition}
\newtheorem{remark}[theorem]{Remark}
\newtheorem*{remark*}{Remark}

\newtheorem{comment2}[theorem]{Comment}

\title[A Master Space for Moduli Spaces of Sheaves]{A Master Space for Moduli Spaces \\ of Gieseker-Stable Sheaves}
\date{\today}
\author[D.~Greb]{Daniel Greb}
  \address{DG: Essener Seminar f\"ur Algebraische Geometrie und Arithmetik, Fakult\"at f\"ur Mathematik, Universit\"at Duisburg--Essen, 45117 Essen, Germany}
  \email{daniel.greb@uni-due.de}
\author[J.~Ross]{Julius Ross}
  \address{JR: Department of Pure Mathematics and Mathematical Statistics, University of Cambridge, Wilberforce Road, Cambridge, CB3 0WB, UK}
  \email{j.ross@dpmms.cam.ac.uk}
\author[M.~Toma]{Matei Toma}
\address{MT: Institut de Math\'ematiques \'Elie Cartan, Universit\'e de Lorraine, B.P. 70239, \newline 54506 Vandoeuvre-l\`es-Nancy Cedex,
France}
\email{matei.toma@univ-lorraine.fr}
\thanks{}
\begin{document}

\begin{abstract}
We consider a notion of stability for sheaves, which we call \emph{multi-Gieseker stability}, that depends on several ample polarisations $L_1, \dots, L_N$ and on an additional parameter $\sigma \in \mathbb{Q}_{\geq 0}^N\setminus\{0\}$. The set of semi stable sheaves admits a projective moduli space $\mathcal M_{\sigma}$.  We prove that given a finite collection of parameters $\sigma$, there exists a sheaf- and representation-theoretically defined master space $Y$ such that each corresponding moduli space is obtained from $Y$ as a Geometric Invariant Theory (GIT) quotient.  In particular, any two such spaces are related by a finite number of ``Thaddeus-flips".  As a corollary, we deduce that any two Gieseker-moduli space of sheaves (with respect to different polarisations $L_1$ and $L_2$) are related via a GIT-master space. This confirms an old expectation and generalises results from the surface case to arbitrary dimension.
\end{abstract}
\thanks{During the preparation of this paper, DG was partially supported by the DFG-Collaborative Research Center SFB/TR 45 ''Periods, moduli spaces and arithmetic of algebraic varieties``. Moreover, he wants to thank Sidney Sussex College, Cambridge, for hospitality and perfect working conditions during a research visit in October 2014. JR is supported by an EPSRC Career Acceleration Fellowship (EP/J002062/1).}
\maketitle


\newcommand{\Space}{\,\,\,\,\,\,\,\,\,\,\,\,\,\,\,\,}

\addtocontents{toc}{\protect\setcounter{tocdepth}{1}}

\addtocontents{toc}{\protect\setcounter{tocdepth}{1}}

This paper continues previous work \cite{GRTI} of the authors in which, building on work of \'Alvarez-C\'onsul and King \cite{ConsulKing}, we re-examine the construction of the moduli space $\mathcal M_L$ of Gieseker-semistable sheaves on a given projective scheme $X$, with an interest in how these spaces change as the polarisation $L$ varies. As a main result, we will prove that for any choice of ample line bundles $L_0$, $L_1$ on $X$, there exists a sheaf- and representation-theoretically defined ``master space'' $Y$ with the property that the moduli spaces $\mathcal M_{L_0}$ and $\mathcal M_{L_1}$ can be obtained from $Y$ as a Geometric Invariant Theory (GIT) quotient with respect to different stability parameters.  In particular, any two such spaces fall into the Variation of Geometric Invariant Theory (VGIT) framework considered by Thaddeus \cite{Thaddeus} and Dolgachev-Hu \cite{DolgachevHu}.

The variation question for moduli of sheaves has attracted quite some attention in the past, and a precise picture has emerged in the case of surfaces (see for example \cite{MatsukiWentworth}), which has recently also been reconsidered from the point of view of Bridgeland stability conditions \cite{BertramMartinez}. However, while a general master space result has been expected by many experts, the approaches used in the surface case fail, and the higher-dimensional situation until recently appeared to be quite mysterious; cf.~the discussion of irrational walls in the introduction of \cite{Schmitt}.

Before stating our results precisely, it is worth commenting on why such a statement does not follow immediately from Simpson's GIT-construction of $\mathcal M_L$ given in \cite{Simpson}.   To do so, we first recall this construction in the case of a single ample line bundle $L_0$, which starts by finding a space with a group action, whose orbits correspond to isomorphism classes of semistable sheaves.  To this end one argues as follows: the family of sheaves of a given topological type that are Gieseker-semistable with respect to $L_0$ is bounded, and thus for $n$ sufficiently large, the evaluation map realises every such sheaf $E$ as a quotient
$$H^0(E\otimes L_0^n) \otimes L_0^{-n} \to E\to 0.$$
Picking an identification between $H^0(E\otimes L_0^n)$ and a fixed vector space of the appropriate dimension, i.e., picking a basis in $H^0(E\otimes L_0^n)$, we can thus think of $E$ as a point in a certain Quot scheme $Q_{L_0}$.   This space has a group action, coming from the possible choices of basis, and the tools of GIT are then applied to an appropriate subscheme $R_{L_0}\subset Q_{L_0}$ to construct the desired moduli space $\mathcal M_{L_0}$.  Doing so requires a choice of linearisation for the group action on $R_{L_0}$, and it is a theorem \cite[Theorem 4.3.3]{Bible} that, for $n$ sufficiently large,  there is such a choice for which GIT-stability agrees with Gieseker-stability.

From this description, it is clear that there is a problem in using VGIT to relate two moduli spaces $\mathcal{M}_{L_0}$ and $\mathcal{M}_{L_1}$, as the moduli space $\mathcal M_{L_1}$ is constructed as the quotient \emph{of a different space} $R_{L_1}$ and not from the space $R_{L_0}$.  It would be nice if one could also produce $\mathcal M_{L_1}$ as a quotient of $R_{L_0}$ with respect to some other linearisation,  but it is not clear whether this is in fact possible.

Our approach is to consider a new kind of stability of sheaves that is a convex combination of Gieseker-stability with respect to different polarisations.  Precisely, let $X$ be a projective manifold and fix a finite collection of ample line bundles $L_j$ on $X$ for $1\le j\le j_0$. Furthermore, suppose that $\sigma=(\sigma_1,\ldots,\sigma_{j_0})$ is a non-zero vector of non-negative rational numbers.  We shall say a torsion-free coherent sheaf $E$ on $X$ is \emph{multi-Gieseker-semistable} (or simply \emph{semistable}) with respect to this data if for all non-trivial proper subsheaves $F\subset E$ the inequality 
\begin{equation*}
 \frac{\sum_{j} \sigma_j \chi(F\otimes L_{j}^m)}{\rank(F)} \le  \frac{\sum_{j} \sigma_j \chi(E\otimes L_{j}^m)}{\rank(E)}\label{eq:introstability}
 \end{equation*}
holds for all $m$ sufficiently large. On the one hand, note that by setting all but one of the components of $\sigma$ to zero we recover the classical notion of Gieseker-semistability. Hence, the notion of multi-Gieseker-stability formally allows us to interpolate between Gieseker-stability with respect to several polarisations. On the other hand, in \cite{GRTI} we prove, under a boundedness hypothesis that we show to hold in many natural setups, that there is a projective moduli space $\mathcal M_{\sigma}$ of sheaves of a given topological type that are semistable with respect to $\sigma$. Analogous to the classical case, this moduli space parametrises $S$-equivalence classes of semistable sheaves.  

The goal of the present paper is to show that any two such spaces can be constructed as a GIT-quotient from the same master space.  In the following, suppose $L_1,\ldots,L_{j_0}$ are a fixed collection of ample line bundles on $X$ and that $\sigma^{{(i)}} = (\sigma_1^{(i)},\ldots,\sigma_{j_0}^{(i)})$ for $i=1,\ldots, i_0$ is a finite collection of non-zero vectors in $\mathbb Q^{j_0}_{\ge 0}$.  We work always under the assumption that the set of sheaves of the given topological type that are semistable with respect to $\sigma^{(i)}$ for some $i$ is bounded. We refer the reader to Section \ref{sec:boundedness} for a brief discussion of some related boundedness results that are strong enough for most applications, including the variation question for Gieseker-moduli spaces considered above.

\begin{theorem*}[Master-Space Construction, Theorem \ref{thm:master}]
 There exists an affine master space $Y$ with an action by a product $G$ of general linear groups such that the following holds: for any $i=1,\ldots,i_0$ the moduli space $\mathcal M_{\sigma^{(i)}}$ is given as the GIT-quotient
$$\mathcal M_{\sigma^{(i)}} = Y \hq _{\theta^{(i)}} G,$$
where $\theta^{(i)}$ is certain character of $G$ used to define GIT-stability.
\end{theorem*}

Thus, we see that the $\mathcal M_{\sigma^{(i)}}$ are all related by so-called \emph{Thaddeus-flips} that arise from VGIT, see \cite{Thaddeus}. In fact, the master space $Y$ itself has a natural sheaf-theoretic meaning and connects the moduli problems considered to the representation theory of a certain quiver. For more details on this modular interpretation of $Y$ the reader is referred to Sections~\ref{sec:quiver}, \ref{subsect:sheaves_into_representations}, and \ref{sect:master_construction}. By setting $j_0=2$ and considering two special parameters $\sigma$, the preceding result immediately implies the following:

\begin{corollary*}[Mumford-Thaddeus principle for Gieseker-moduli spaces, Corollary~\ref{cor:master}]
Let $L_0$ and $L_1$ be ample line bundles on $X$.  Then, the moduli spaces $\mathcal M_{L_0}$ and $\mathcal M_{L_1}$ of sheaves of a given topological type that are Gieseker-semistable with respect to $L_0$ and $L_1$, respectively, are related by a finite number of Thaddeus-flips.
\end{corollary*}
Again, this is only a weak formulation of the outcome of our investigation, and the master space used is closely connected to the collection of moduli problems at hand.

It is worth emphasising that the above theorem was proved in \cite{GRTI} under the additional assumption that all the $\sigma^{(i)}_j$ are strictly positive.  This suffices to conclude the above corollary under the additional assumption that we are dealing with torsion-free sheaves on a smooth threefold, see \cite[Theorem~12.1]{GRTI}, or that $L_0$ and $L_1$ are ``general" polarisations in a suitable sense, see \cite[Theorem~4.4]{GRTIIpreprint}, although the argument for these results is somewhat indirect. Thus, the real purpose of this paper is to address the degenerate case when some of the $\sigma^{(i)}_j$ vanish, thus implementing the naive idea behind the introduction of the notion of multi-Gieseker-stability and allowing us to conclude the corollary in the stated generality.  We also emphasise that the assumption that $X$ be smooth and that $E$ be torsion-free is not necessary, and will be relaxed below.

In \cite{GRTIIpreprint}, we developed a general approach in order to realise the intermediate spaces that appear in a VGIT-passage from one Gieseker-moduli space to another as moduli of sheaves, and executed it in dimension three. In principle, given the main theorem of the present paper, this approach could be pushed through also in higher dimensions; however, the required Riemann-Roch computations seem to be growing in complexity rather quickly. For more information, the reader is refereed to Section \ref{sec:uniform} as well as to said paper \cite{GRTIIpreprint}. It remains an interesting open problem to derive analogous more refined results concerning the intermediate steps in the VGIT-passage even for special base manifolds $X$, low rank, and special choice of Chern-classes.

\subsubsection*{Global Assumption}
We always work over an algebraically closed field $k$ of characteristic zero.

\section{Preliminaries on multi-Gieseker stability and quiver representations}\label{sec:prelim}

For convenience of the reader, and to make this paper reasonably self-contained, we recall some preliminaries and definitions we need from \cite{GRTI} and sketch the construction of the moduli space of multi-Gieseker-stable sheaves.

\subsection{Geometric Invariant Theory} For standard concepts from Geometric Invariant Theory the reader is referred to \cite{MumfordGIT} as well as \cite{SchmittBook}.  When a group $G$ acts on a scheme $Y$ we shall write $Y\hq G$ for the good quotient (when it exists).  We shall use the following result (which is \cite[Exercise 1.5.3.3]{SchmittBook}) concerning quotients by product group actions. 

\begin{lemma}\label{lem:exerciseschmitt}
Suppose $G'$ and $G''$ are reductive linear algebraic groups.  Let $G'\times G''$ act on a scheme $Z$ of finite type over $k$ and suppose this action admits a good quotient $Z\hq (G'\times G'')$.  Then, the induced action of $G''$ on $Z$ admits a good quotient $Z\hq G''$, which is endowed with an induced $G'$-action, and $Z\hq (G'\times G'')$ is a good quotient of $Z\hq G''$ by $G'$; i.e., we obtain an isomorphism
$$ Z\hq (G'\times G'') \cong (Z\hq G'') \hq G'.$$
\end{lemma}
\begin{proof}
First we recall the following well-known statement: suppose that $G''$ acts on schemes $Y$ and $Y'$ of finite type, and assume that $Y'$ admits a good quotient $Y'\hq G''$.  Then if $Y\to Y'$ is a $G''$-equivariant affine morphism, then $Y$ also admits a good quotient $Y\hq G''$, and the induced map $Y\hq G'' \to Y'\hq G''$ is affine (this statement is proved for example in \cite[Lemma 5.1]{Ramanathan} by first working affine locally and then patching).   We apply this to the morphism $Z\to Z\hq (G\times G'')$, which is possible as this map is affine and $G''$-invariant, and as $Z\hq (G\times G'')$  certainly admits a good quotient by the (trivial) action of $G''$.  This gives an affine $Z\hq G'' \to Z\hq (G'\times G'')$, which is easily checked to satisfy the requirements of being the good quotient of $Z\hq G''$ by $G'$.
\end{proof}

\subsection{Preliminaries on sheaves} \label{sec:prelimsheaves}   Throughout this paper, $X$ shall be a projective scheme over an algebraically closed field of characteristic zero.    The \emph{dimension} of a coherent sheaf $E$ on $X$ is the dimension of its support $\{ x\in X : E_x\neq 0\}$, and we say that $E$ is \emph{pure} of dimension $d$ if all non-trivial coherent subsheaves $F\subset E$ have dimension $d$.  Let $\tau$ be an element of $B(X)_\Q:=B(X)\otimes_\Z\Q$, where $B(X)$ is the group of cycles on $X$ modulo algebraic equivalence, see \cite[Def.~10.3]{F}.    We say that 
$E$ is \emph{of topological type} $\tau$ if its homological Todd class $\tau_X(E)$ equals  $\tau$.   Given a very ample line bundle $L$ on $X$ we say $E$ is \emph{n-regular} with respect to $L$ if $H^i(E\otimes L^{n-i}) =0$ for all $i>0$.  When dealing with several line bundles, the following definition is convenient: 

\begin{definition}
Suppose that $\underline{L}= (L_1,\ldots,L_{j_0})$, where each $L_j$ is a very ample line bundle on $X$.  We say that a coherent sheaf $E$ is \emph{$(n,\underline{L})$-regular} if $E$ is $n$-regular with respect to $L_j$ for all $j \in \{1, \ldots, j_0\}$. 
\end{definition}

\subsection{Multi-Gieseker stability}  A \emph{stability parameter} $\sigma = (\underline{L},\sigma_1,\ldots,\sigma_{j_0})$ consists of $\underline{L} = (L_1,\ldots,L_{j_0})$ for some ample line bundles $L_j$ on $X$, and  $\sigma_j\in \mathbb R_{\ge 0}$ such that not all the $\sigma_j$ are zero.  We say that $\sigma$ is \emph{rational} if all the $\sigma_j$ are rational.  We say $\sigma$ is \emph{positive} if $\sigma_j>0$ for all $j$; otherwise $\sigma$ is said to be \emph{degenerate}.

In the subsequent discussion the vector $\underline{L}$ will be fixed, so by abuse of notation we will sometimes confuse $\sigma$ and the vector $(\sigma_1,\ldots,\sigma_{j_0})$.    Thus, we allow $\sigma$ to vary in a subset of $(\mathbb R_{\ge 0})^{j_0}\setminus\{0\}$.  We emphasise that whereas we allow the $\sigma_j$ to be irrational, we will always assume that the $L_j$ are genuine (integral) line bundles. The \emph{multi-Hilbert polynomial} of a coherent sheaf $E$ of dimension $d$ with respect to such a stability parameter $\sigma$ is
   \begin{equation*}
     P_E^\sigma(m) := \sum\nolimits_{j} \sigma_j \chi(E\otimes L_j^m) =\sum\nolimits_{i=0}^d \alpha^{\sigma}_i(E) \frac{m^i}{i!},
   \end{equation*}
where $\alpha^{\sigma}_i(E)$ are uniquely defined coefficients, cf.~\cite[p.~10]{Bible}.  We denote the leading one by $r^{\sigma}_E:=\alpha^\sigma_d(E)$ and note that it is strictly positive by the hypothesis on $\sigma$.

\begin{definition}[Multi-Gieseker-stability]
The \emph{reduced multi-Hilbert polynomial} of a coherent sheaf of dimension $d$ is defined to be
$$ p_E^{\sigma}(m) := \frac{ P_E^{\sigma}(m)}{r^{\sigma}_E}.$$
   We say that a coherent sheaf $E$ is \emph{multi-Gieseker-(semi)stable} with respect to this data if it is pure and if for all non-trivial proper subsheaves $F\subset E$ we have 
  \begin{equation*}
p_F^{\sigma} (\le) p_E^{\sigma}.
\end{equation*}
\end{definition}
When convenient, we will refer to this notion as being \emph{$\sigma$-semistable} or merely \emph{semistable} in case no confusion may arise.  Of course, if $\sigma =\underline{e}_i = (0,\ldots,1,\ldots,0)$ is the standard basis vector, then $E$ is (semi)stable with respect to $\sigma$ if and only if it is Gieseker-(semi)stable with respect to $L_i$.      The notion of multi-Gieseker-stability comes with the standard notion of \emph{Jordan-H\"older filtration}, and we say two semistable sheaves $E$ are \emph{S-equivalent} if the graded modules associated with the respective filtrations are isomorphic.

\subsection{Boundedness}\label{sec:boundedness}
A set $\mathcal S$ of isomorphism classes of coherent sheaves on $X$ is said to be \emph{bounded} if there exists a scheme $S$ of finite type and a coherent $\mathcal O_{S\times X}$-sheaf $\mathcal E$ such that every $E\in \mathcal S$ is isomorphic to $\mathcal E_{\{s\}\times X}$ for some closed point $s\in S$.    From \cite[Lemma 1.7.6]{Bible}, we know that the set of $(n,\underline{L})$-regular sheaves of a given topological type is bounded.  Conversely, it follows from Serre Vanishing that, if $\mathcal S$ is a bounded family of sheaves, then for $n \gg 0$ each $E\in \mathcal S$ is $(n,\underline{L})$-regular.

Throughout this paper, we will work under the hypothesis that the set of sheaves of a given topological type that are semistable with respect to any of the stability parameters under consideration at that time is bounded.  If $X$ is smooth, and if the sheaves in question are torsion-free, we proved in \cite[Corollary 6.12]{GRTI} that this hypothesis holds if either (a) the rank of the sheaves in question is 2 or (b) the Picard rank of $X$ is at most 2 or (c) the dimension of $X$ is at most 3.   We also prove \cite[Theorem 6.8]{GRTI} that this boundedness hypothesis holds for torsion-free sheaves on a manifold of dimension $d$ as long as the stability parameters $\sigma = (L_1,\ldots,L_{j_0},\sigma_1,\ldots,\sigma_{j_0})$ in question have the property that $\sum_j \sigma_j c_1(L)^{d-1}$ lies in a given compact subset of the image of the ample cone under the map $N^1(X)_{\mathbb R}\to N_1(X)_{\mathbb R}$, $x\mapsto x^{d-1}$.  This latter statement is usually adequate to ensure that the desired boundedness holds for studying the problem of variation of moduli spaces of Gieseker-semistable sheaves, see also the discussion preceding the proof of the Mumford-Thaddeus principle for Gieseker-moduli spaces in Section~\ref{sect:master_construction}.

\subsection{Quivers and their representations }\label{sec:quiver}
Following the ideas of \'Alvarez-C\'onsul--King presented in \cite{ConsulKing} we now discuss how to embed the category of sufficiently regular sheaves into the category of representations for certain quivers.   We will use the standard notations used in representation theory of quivers, as fixed for example in \cite[Section~3]{King}.   We denote by $\mathbf{Vect}_k$ the category of vector spaces over a field $k$.  Given a $j_0\in \mathbb N^+$ we define a labelled quiver  $\mathcal{Q} = (\mathcal{Q}_0, \mathcal{Q}_1, h,t\colon \mathcal{Q}_1 \to \mathcal{Q}_0, H\colon \mathcal{Q}_1 \to \mathbf{Vect}_k)$ as follows.  Let $ \mathcal{Q}_0 := \{v_i, w_j \mid i,j = 1, \dots, j_0\}$ be a set of pairwise distinct vertices, and $\mathcal{Q}_1 := \{\alpha_{ij} \mid i,j= 1, \dots, j_0 \}$
the set of arrows, whose heads and tails are given by $h (\alpha_{ij}) = w_j$ and $t (\alpha_{ij}) = v_i$.  The arrows will each be labelled by a vector space, encoded by a function $H\colon \mathcal{Q}_1 \to \mathbf{Vect}_k$ written as $H(\alpha_{ij}) = H_{ij}$, which will be fixed later.  This quiver can be pictured as follows (where for better readability we restrict to the case $j_0=3$): \enlargethispage{0.4cm}
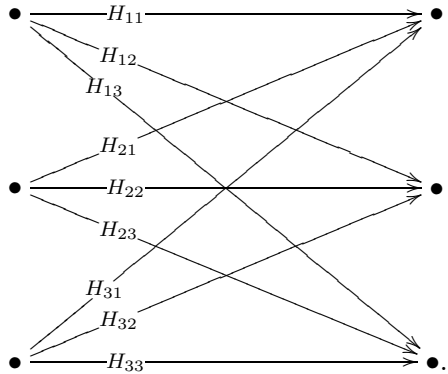
\[\begin{xymatrix}
{ \bullet \ar[rrrrr]|<<<<<<<<<<<<{H_{11}}\ar[rrrrrdd]|<<<<<<<<<<<<{H_{12}}\ar[rrrrrdddd]|<<<<<<<<<<<<{H_{13}}&  & & &  &\bullet \\
                                        &  &  & &  &     \\
  \bullet \ar[rrrrruu]|<<<<<<<<<<<<{H_{21}} \ar[rrrrr]|<<<<<<<<<<<<{H_{22}} \ar[rrrrrdd]|<<<<<<<<<<<<{H_{23}}& & & & & \bullet \\
                            &   & &  &   &    \\
  \bullet \ar[rrrrruuuu]|<<<<<<<<<<<<{H_{31}} \ar[rrrrr]|<<<<<<<<<<<<{H_{33}} \ar[rrrrruu]|<<<<<<<<<<<<{H_{32}}& & & & &\bullet    .
}
  \end{xymatrix}
\]

A \emph{representation} $M$ of $\mathcal{Q}$ over a field $k$ is a collection $V_i, W_j, i,j= 1, \ldots, n$ of $k$-vector spaces together with $k$-linear maps $\phi_{ij}\colon V_i \otimes H_{ij} \to W_j$.   By abuse of notation we write this as $M = \bigoplus_j V_j \oplus W_j$, where it is understood that $V_j$ (resp. $W_j$) is associated with the vertex $v_j$ (resp. $w_j$) of $\mathcal{Q}$.  A \emph{subrepresentation} is a collection of subspaces $V_i'\subset V_i$ and $W_i'\subset W_i$ such that $\phi_{ij}(V_i'\otimes H_{ij}) \subset W_{j}'$.  We call the vector $$\underline{d} := (\dim V_1, \dim W_1,\dots, \dim V_{j_0}, \dim W_{j_0})=: (d_{11}, d_{12}, \ldots, d_{{j_0 1}}, d_{{j_0 2}})$$ the \emph{dimension vector} of $M$.   The abelian category of representation of $\mathcal{Q}$ is equivalent to the abelian category of $A$-modules, where $A$ is the (finite-dimensional) path algebra of $\mathcal{Q}$, cf.~\cite{King}.  Henceforth, will refer to a representation of $\mathcal{Q}$ as an \emph{$A$-module}.

\subsection{Families} \label{sec:families} Let $S$ be a scheme.   A \emph{flat family $\mathscr{M}$ of $A$-modules over $S$} is a sheaf $\mathscr{M}$ of modules over the sheaf of algebras $\mathscr{A} := \mathcal{O}_S \otimes A$ on $S$ that is locally free as a sheaf of $\mathcal{O}_S$-modules. In other words, we have a decomposition $\mathscr{M}=\bigoplus_j \mathscr V_j\oplus \mathscr W_j$, where $\mathscr V_j$ and $\mathscr W_j$ are vector bundles on $S$ associated with the vertices $v_j,w_j$ of $\mathcal{Q}$, together with sheaf morphisms $\mathscr{V}_i\otimes H_{ij} \to \mathscr W_j$ associated with the arrows of $\mathcal{Q}$. By a \emph{trivialisation} of such an $\mathscr{M}$ we mean a trivialisation of each $\mathscr V_j$ and $\mathscr W_j$.
The \emph{representation space} 
\begin{equation}
R:= \mathrm{Rep}(\mathcal{Q}, \underline{d}) :=  \bigoplus\nolimits_{i,j=1}^{j_0} \mathrm{Hom}_k (k^{d_{i1}} \otimes H_{ij}, k^{d_{j2}} ) \label{eq:repspace}\end{equation}
parametrises representations of $\mathcal{Q}$. The reductive group $G:=  \prod\nolimits_{j=1}^{j_0} \bigl(GL_k(d_{j1}) \times GL_k(d_{j2}) \bigr)$
acts linearly on $R$ by conjugation, and the orbits correspond to isomorphism classes of $A$-modules of dimension $\underline{d}$.   We denote the tautological flat family of $A$-modules over $R$ by $\mathscr{U}$, which we observe comes with a canonical trivialisation. Note that a trivialisation of a flat family $\mathscr{M}$ of $A$-modules over a scheme $S$ gives a canonical morphism $\sigma\colon S\to R$ and isomorphism $\sigma^* \mathscr{U} \cong \mathscr{M}$, and a different choice of trivialisation changes $\sigma$ by the action of a unique morphism $S\to G$.

\subsection{Embedding categories of sheaves into categories of quiver representations}\label{subsect:sheaves_into_representations}
To connect the preceding discussion with the moduli theory of sheaves,  given line bundles $L_j$ on $X$ for $j=1,\ldots,j_0$ and integers $m > n$, we let $H_{ij} := \Hom(L_j^{-m}, L_i^{-n})$.  Given a coherent sheaf $E$ we consider the representation of $\mathcal{Q}$ in which $V_i= H^0(E\otimes L_i^n)$ and $W_j= H^0(E\otimes L_j^m)$ and $\phi_{ij}\colon H^0(E\otimes L_i^n)\otimes H_{ij} \to H^0(E\otimes L_j^m)$ are the natural multiplication maps.  Defining a sheaf $T := \bigoplus\nolimits_{j= 1}^{j_0} L_j^{-n}  \oplus L_j^{-m}$,  we denote this representation by 
\[\Hom(T,E) = \bigoplus\nolimits_{j} H^0(E\otimes L_j^{n}) \oplus  H^0(E\otimes L_j^{m}).\]
In the set-up at hand, the path algebra is isomorphic to the finite-dimensional $k$-algebra $A = L \oplus H \subset \mathrm{End}_X(T)$
generated by the projection operators onto the summands $L_i^{-n}$ and $L_j^{-m}$ of $T$ (collected in the subalgebra $L$) and $H = \bigoplus\nolimits_{i,j} H_{ij}$ (see \cite[Section~5.1]{GRTI} for more details). Then, the $\mathcal{Q}$-representation $\Hom(T,E)$ has the obvious structure of an $A$-module. The next theorem says that this way of associating an $A$-representation to a sheaf does not lose any information if the sheaf in question is $n$-regular and $m$ is big enough.

\begin{theorem}[Embedding regular sheaves into the category of representations of $\mathcal{Q}$, {\cite[Theorem~5.7]{GRTI}}]\label{thm:categoryembedding}
Suppose that $\underline{L}= (L_1,\ldots,L_{j_0})$, where each $L_j$ is a very ample line bundle on $X$. Given $n\in \mathbb N$, for all $m \gg n$, the functor 
$$ E \mapsto \Hom(T,E)$$
is a fully faithful embedding from the category of $(n,\underline{L})$-regular sheaves of topological type $\tau$ to the category of representations of $\mathcal{Q}$.  In other words, if $E$ is an $(n,\underline{L})$-regular sheaf of topological type $\tau$, the natural evaluation map $\varepsilon_E\colon \Hom(T, E) \otimes_A T \to E$ is an isomorphism. 
\end{theorem}

We remark that part of the proof of the above theorem is the construction of the adjoint functor of $E\mapsto \Hom(T,E)$, which above and henceforth is denoted $M\mapsto M\otimes_A T$.  In the following, for a sheaf $\mathcal{E}$ on $X\times S$ and a sheaf $\mathscr{M}$ of $A$-modules on $S$ we shall use the abbreviations 
\[\sheafHom_X(T, \mathcal{E}):= p_{S*} \bigl(\sheafHom_{X \times S}(p_X^*T, \mathcal{E}) \bigr) \cong p_{S*} \bigl(\mathcal{E} \otimes_{\mathcal{O}_{X \times S}} T^\vee  \bigr)\]
and 
\[\mathscr{M} \otimes_\mathscr{A} T := p_S^* \mathscr{M} \otimes_{\mathscr{A}} p_X^* T, \] 
where $\mathscr{A} := \mathcal{O}_S \otimes A$, and $p_S\colon X \times S \to S$ resp.~$p_X\colon X \times S \to X$ are the canonical projections.  The previous  result, Theorem \ref{thm:categoryembedding}, extends in a natural way to flat families of sheaves and modules (see \cite[Proposition 5.8]{GRTI}  for the precise statement).   In the next step, we identify the image of the embedding functor. For this, we let $P_j(k) = \chi(E\otimes L_j^k)$ be the Hilbert-polynomial for a (resp.\ any) coherent sheaf $E$ of topological type $\tau$ with respect to $L_j$.

\begin{proposition}[Identifying the image of the embedding functor, {\cite[Proposition~5.9]{GRTI}}]\label{prop:identifyingtheimage}
For $m\gg n$ as in Theorem~\ref{thm:categoryembedding} the following holds: If $B$ is any Noetherian scheme and $\mathscr{M}$ a $B$-flat family of $A$-modules of dimension vector $$\underline{d} = \bigl(P_1(n), P_1(m), \ldots, P_{j_0} (n), P_{j_0} (m)\bigr),$$ then there exists a (unique) locally closed subscheme $\iota\colon B^{[reg]}_{\tau} \hookrightarrow B$
with the following properties.
\begin{enumerate}
 \item[(a)] $\iota^* \mathscr{M} \otimes_\mathscr{A} T$ is a $B^{[reg]}_{\tau}$-flat family of $(n,\underline{L})$-regular sheaves on $X$ of topological type $\tau$, and the unit map
 \[\eta_{\iota^*\mathscr{M}}\colon  \iota^* \mathscr{M} \to \sheafHom_X(T, \iota^* \mathscr{M} \otimes_\mathscr{A} T)\]
is an isomorphism.
 \item[(b)] If $\sigma\colon S \to B$ is such that $\sigma^* \mathscr{M} \cong \sheafHom_X (T,\mathcal{E})$ for an $S$-flat family $\mathcal{E}$ of $(n,\underline{L})$-regular sheaves on $X$ of topological type $\tau$, then $\sigma$ factors through $\iota\colon B^{[reg]}_{\tau} \hookrightarrow B$ and $\mathcal{E} \cong \sigma^*\mathscr{M} \otimes_\mathscr{A} T$.
\end{enumerate}
\end{proposition}

\subsection{Preservation of stability}  We now explain how the functor $T\mapsto \Hom(T,E)$ preserves stability.  For this we need King's notion of stability for $A$-modules $M$ of dimension vector $\underline{d} = (d_{11}, d_{12}, \ldots, d_{{j_0 1}}, d_{{j_0 2}})$.  Fix  $\sigma = (\sigma_1,\ldots,\sigma_{j_0})$ with $\sigma_j\in \mathbb Q_{\ge 0}$ not all equal to zero.  Define a vector $\theta_\sigma = (\theta_{11}, \theta_{12}, \ldots, \theta_{{j_0 1}}, \theta_{{j_0 2}})$ by
\begin{equation*}
 \theta_{j1} := \frac{\sigma_j}{\sum_i \sigma_i d_{i1}}\text{ and }\theta_{j2} := \frac{- \sigma_j}{\sum_i \sigma_i d_{i2}}\text{ for }j= 1, \ldots, j_0.
\end{equation*}
For an $A$-module $M' = \bigoplus V_j'\oplus W_j'$, we set
\begin{equation*}
 \theta_\sigma(M') := \sum\nolimits_{j} \theta_{j1} \dim V_j' + \sum\nolimits_{j} \theta_{j2} \dim W_j',
\end{equation*}
which makes $\theta_{\sigma}$ an additive function from the set $\mathbb{N}^{2 j_0}$ of possible dimension vectors to $\mathbb R$.  Note that if $M$ is an $A$-module of dimension vector $\underline{d}$, then  $\theta_{\sigma}(M)=\sum_j (\theta_{j1} d_{j1} + \theta_{j2} d_{j2})=0$.

\begin{definition}[Semistability for $A$-modules]\label{def:semistabilitymodule}
Let $M$ be an $A$-module with dimension vector $\underline{d}$.  We say that $M$ is \emph{(semi)stable (with respect to $\sigma$)} if for all non-trivial proper submodules $M'\subset M$ we have $\theta_\sigma(M')(\le)0$.  We let $R^{\sigma\text{-ss}}$ denote the open subset of $R = \mathrm{Rep}(\mathcal{Q}, \underline{d}) $ consisting of modules that are semistable with respect to $\sigma$.   
\end{definition} 

Every $\sigma$-semistable $A$-module has a Jordan-H\"older filtration with respect to $\theta_\sigma$, cf.~\cite[p.~521/522]{King}, and we call two modules \emph{$S$-equivalent} if the graded modules associated with the respective filtrations are isomorphic.  The following is the main technical result of \cite{GRTI}, for which, as always, we presume that the set of sheaves of topological type $\tau$ that are semistable with respect to $\sigma$ is bounded.

\begin{theorem}[Comparison of semistability and JH filtrations,{\cite[Theorem 8.1]{GRTI}}]\label{thm:mainsemistabilitycomparison} 
 For all integers $m\gg n\gg p\gg 0$ the following holds  for any sheaf $E$ on $X$ of topological type $\tau$:
 \begin{enumerate}
 \item  $E$ is semistable if and only if it is pure, $(p,\underline{L})$-regular, and $\Hom(T,E)$ is semistable.
 \item Suppose $\sigma$ is positive.  If $E$ is semistable, then
$$ \Hom(T,gr E)\cong gr \Hom(T,E),$$
where $gr$ denotes the graded object coming from a Jordan-H\"older filtration of $E$ or $\Hom(T,E)$, respectively.  In particular, two semistable sheaves $E$ and $E'$ are $S$-equivalent if and only if $\Hom(T,E)$ and $\Hom(T,E')$ are $S$-equivalent.  
 \end{enumerate}
\end{theorem}

\subsection{Construction of the moduli space}\label{sec:construction}
We sketch now how the above result is used to construct the moduli space we are looking for.  Assume $\sigma$ is a positive rational stability parameter, and that the set of sheaves of topological type $\tau$ that are semistable with respect to $\sigma$ is bounded.      We choose natural numbers $p,n,m \in \mathbb{N}$ such that Theorem~\ref{thm:mainsemistabilitycomparison} holds.   Moreover, by increasing $m$ if necessary, we may assume that the conclusion of Theorem~\ref{thm:categoryembedding} holds as well.  Note that by assumption, every semistable sheaf $E$ of topological type $\tau$ is  $(p,\underline{L})$-regular, and therefore also $(n, \underline{L})$- and $(m, \underline{L})$-regular. 

Recall $P_j$ denotes the Hilbert polynomial of $E$ with respect to $L_j$, where $E$ is a (any) sheaf of topological type $\tau$.  We consider the dimension vector  
$$\underline{d} = \bigl(P_1(n), P_1(m), \ldots, P_{j_0} (n), P_{j_0} (m)\bigr),$$ 
and as above let
\begin{equation}\label{eq:repvariety:repeat}
R:= \mathrm{Rep}(\mathcal{Q}, \underline{d}) =  \bigoplus\nolimits_{i,j=1}^{j_0} \mathrm{Hom}_k (k^{P_i(n)} \otimes H_{ij}, k^{P_j(m)} ) \end{equation}
be the representation space of the quiver $\mathcal{Q}$ corresponding to the dimension vector $\underline{d}$. Here, as before we have used the notation $H_{ij} = H^0(L_i^{-n}\otimes L_j^m)$.   

Let $R_{\tau}^{[n\text{-reg}]}\subset R$ be the locally closed subscheme parametrising those modules that are in the image of the $\Hom(T, -)$-functor on the category of $(n,\underline{L})$-regular sheaves, whose existence is guaranteed by Proposition~\ref{prop:identifyingtheimage}. Inside this, there is an open subscheme $Q \subset R_{\tau}^{[n\text{-reg}]}$ parametrising those modules whose associated sheaf is not only $(n,\underline{L})$- but actually $(p,\underline{L})$-regular, and a further open subscheme
$$Q^{[\sigma\text{-ss}]} \subset Q \subset  R_{\tau}^{[n\text{-reg}]}$$
parametrising representations $\Hom(T,E)$, where $E$ is $\sigma$-semistable and of topological type $\tau$. The reductive group
\begin{equation}\label{eq:def_of_group}
 G:=  \prod\nolimits_{j=1}^{j_0} \bigl(GL_k(P_j(n)) \times GL_k(P_j(m)) \bigr)
\end{equation}
acts linearly on $R$ by conjugation, and $R^{[n\text{-}reg]}_{\tau}$ and $Q^{[\sigma\text{-}ss]}$ are preserved by the $G$-action.

\begin{theorem}[Existence of moduli spaces for $\sigma$-semistable sheaves,  {\cite[Theorems 9.4 and 10.1]{GRTI}}]\label{thm:moduliexist}
Assume that $\sigma$ is a positive stability parameter such that the set of sheaves of topological type $\tau$ that are semistable with respect to $\sigma$ is bounded, and let $Z:= \overline{Q^{[\sigma\text{-ss}]}}$ be the scheme-theoretic closure in $R$.  Then,
$$ Z^{\sigma\text{-ss}} : = Z\cap R^{\sigma\text{-ss}} = Q^{[\sigma\text{-ss}]}.$$
Moreover,
$$\mathcal M_{\sigma} := Z^{\sigma\text{-ss}}\hq G$$
exists as a good quotient, and is the moduli space of sheaves of topological type $\tau$ that are semistable with respect to $\sigma$.
\end{theorem}

The proof of Theorem~\ref{thm:moduliexist} rests on our comparison between stability of a sheaf $E$ and stability of the corresponding $A$-module $\Hom(T,E)$. Since we are assuming $\sigma$ to be positive, Theorem \ref{thm:mainsemistabilitycomparison}(2) also allows us to compare $S$-equivalence classes before and after the categorical embedding. This gives us enough control on the closure of $G$-orbits to ensure that the good quotient $Z^{\sigma\text{-ss}}\hq G$ exists. We think of the quotient $Z^{\sigma\text{-ss}}\hq G$ as being embedded inside the good quotient $R^{\sigma\text{-ss}}\hq G$ of semistable $A$-modules constructed by King. We refer the reader to \cite[Section 9]{GRTI} for a more detailed discussion, in particular, for a description of the moduli functor that $\mathcal M_{\sigma}$ corepresents.

\begin{remark}
The above discussion requires that each $\sigma_j$ is rational.  However, we show in \cite[Corollary 4.4]{GRTI} that, in the case of torsion-free sheaves on an integral scheme, this is enough to also prove the existence of $\mathcal M_{\sigma}$ even if some of the $\sigma_j$ are irrational.
\end{remark}

\section{Degenerate stability parameters}\label{sec:degenerate}

\subsection{Setup} \label{sect:setup}
Let $\sigma = (\sigma_1,\ldots,\sigma_{j_0})\in \mathbb Q_{\ge 0}^{j_0}\setminus\{0\}$.   In the previous section we sketched the construction of the moduli space $\mathcal M_{\sigma}$ when $\sigma$ is positive. When $\sigma$ is degenerate, in order to show the existence of the desired moduli space, suppose we
ignore all those $j$ such that $\sigma_j=0$. Since this does not change the (multi-Gieseker-)stability condition on sheaves we obtain in this way the moduli space $\mathcal{M}_\sigma$, but at a cost of modifying the parameter space used in the GIT construction. In particular, this approach does not directly yield the desired master space that can be used for both positive and degenerate stability conditions at the same time. For this reason we instead take a closely related but different approach. 

As above, $X$ is a projective scheme over $k$, and $L_1,\ldots,L_{j_0}$ are very ample line bundles on $X$.   We also fix a topological type $\tau$.    Renaming indices if necessary, we may suppose there is a $j'<j_0$ such that $\sigma_j>0$ for all $j\le j'$ and $\sigma_j=0$ for all $j>j'$. Consider the truncation
$$\sigma' := (L_1,\ldots, L_{j'}, \sigma_1,\ldots, \sigma_{j'}),$$
which is a positive stability parameter, and set $\underline{L}'= (L_1, \dots, L_{j'})$.    Our aim is to compare the GIT setup for $\sigma$ with that for $\sigma'$.

 To this end,  let $\mathcal{Q}$ be the quiver with vertices $\{v_i,w_i \mid i,j=1,\ldots, j_0\}$ discussed in Section \ref{sec:quiver}, and let $\mathcal{Q}'$ be the full subquiver with vertices $\{v_i, w_i\mid i,j=1,\ldots,j'\}$.  Pictorially, we think of $\mathcal{Q}'$ as being the top $j'$ rows of $\mathcal{Q}$ together with the induced arrows.     As always, we assume the set of sheaves of topological type $\tau$ that are semistable with respect to $\sigma$ is bounded, and we also choose integers $m\gg n\gg p\gg 0$ so that Theorem \ref{thm:categoryembedding} and Theorem \ref{thm:mainsemistabilitycomparison}(1) hold.    By increasing these integers if necessary, we can assume that Theorem \ref{thm:categoryembedding}  and Theorem \ref{thm:mainsemistabilitycomparison} also hold for the positive stability parameter $\sigma'$.

Now let $T$, $A$, $\underline{d}$, $R$, $Q$, $G$ be the objects associated with $\mathcal{Q}$ defined in Section \ref{sec:prelim}, and let $T'$, $A'$, $\underline{d}'$, $R'$, $Q'$, $G'$ be the corresponding objects associated with $\mathcal{Q}'$.       Thus, for a sheaf $E$, the $A$-module $\Hom(T,E)$ is a representation of $\mathcal{Q}$ and the $A'$-module $\Hom(T',E')$ is a representation of $\mathcal{Q}'$.    If $E$ is $(n,\underline{L})$-regular of topological type $\tau$, then $\Hom(T,E)$ and $\Hom(T',E)$ have dimension vector $\underline{d}$ and $\underline{d}'$ respectively, and so yield elements of $R=\mathrm{Rep}(\mathcal{Q},\underline{d})$ and $R'= \mathrm{Rep}(\mathcal{Q}',\underline{d}')$, respectively.  

\begin{definition}[Projection induced by subquiver]
We let
$$ \pi\colon R\to R'$$
be the projection induced from the inclusion $\mathcal{Q}'\subset \mathcal{Q}$.  
\end{definition}
Pictorially, $\pi$ can be thought of as taking a representation of $\mathcal{Q}$ and throwing away all but the top $j'$ rows; so, when $j_0=3$ and $j'=2$, this looks as follows:

$$\begin{array}{ccc}
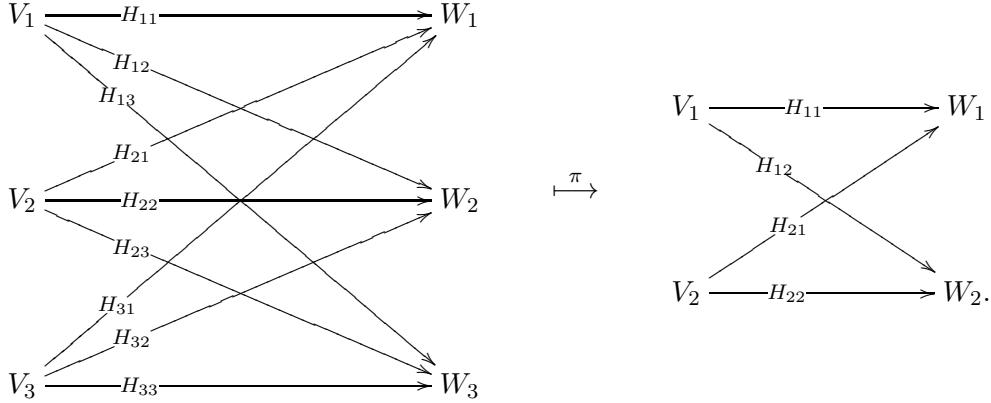
\begin{xymatrix}
{ V_1 \ar[rrrrr]|<<<<<<<<<<<<{H_{11}}\ar[rrrrrdd]|<<<<<<<<<<<<{H_{12}}\ar[rrrrrdddd]|<<<<<<<<<<<<{H_{13}}&  & & &  &W_1 \\
                                        &  &  & &  &     \\
  V_2 \ar[rrrrruu]|<<<<<<<<<<<<{H_{21}} \ar[rrrrr]|<<<<<<<<<<<<{H_{22}} \ar[rrrrrdd]|<<<<<<<<<<<<{H_{23}}& & & & & W_2 \\
                            &   & &  &   &    \\
  V_3 \ar[rrrrruuuu]|<<<<<<<<<<<<{H_{31}} \ar[rrrrr]|<<<<<<<<<<<<{H_{33}} \ar[rrrrruu]|<<<<<<<<<<<<{H_{32}}& & & & &W_3   
}
  \end{xymatrix}
 & \begin{xymatrix}
    {  \\ \\ \quad \stackrel{\pi}{\longmapsto} \quad\\ \\ \\ }
   \end{xymatrix}
 & 
\begin{xymatrix}
{   & & &   \\ V_1 \ar[rrr]|<<<<<<<<<<<<{H_{11}}\ar[rrrdd]|<<<<<<<<<<{H_{12}}&   &  &W_1 \\
                                        &   &  &     \\
  V_2 \ar[rrruu]|<<<<<<<<<<<<{H_{21}} \ar[rrr]|<<<<<<<<<<{H_{22}} &  & & W_2.
}
  \end{xymatrix}
 
\end{array}$$
\begin{remark}[Kronecker quivers as subquivers of $\mathcal{Q}$]
Note that by setting all but one of the rows to zero we obtain the Kronecker quiver used  by \'Alvarez-C\'onsul and King \cite{ConsulKing} in their construction of the moduli space of Gieseker-semistable sheaves. This observation is our original motivation to study degenerate stability parameters.
\end{remark}

As before, let $\theta_{j1}$ and $\theta_{j2}$ be given by
\begin{equation}
 \theta_{j1} := \frac{\sigma_j}{\sum_{i=1}^{j_0} \sigma_{i} d_{i1}} \quad \text{ and } \quad \theta_{j2} := \frac{- \sigma_j}{\sum_{i=1}^{j_0} \sigma_{i} d_{i2}} \quad \text{ for }j= 1, \ldots, j_0,
\end{equation}
and for an $A$-module $M = \bigoplus_j V_j\oplus W_j$ set
\begin{equation}\label{eq:bigthetadef}
 \theta_\sigma(M) := \sum\nolimits_{j} \theta_{j1} \dim V_j + \sum\nolimits_{j} \theta_{j2} \dim W_j.
\end{equation}
So, if $M$ is an $A$-module of dimension vector $\underline{d}$, then $\theta_{\sigma}(M)=0$.  Observe that if we define $\theta'_{j1}$ and $\theta'_{j2}$ in an analogous way, with $\sigma$ replaced by $\sigma'$ and $j_0$ replaced by $j'$, then $\theta'_{ji}=\theta_{ji}$ for all $j\le j'$ and $i=1,2$.    Thus, if $M$ is an $A'$-module of dimension vector $\underline{d}'$, then $\theta_{\sigma'}(M)=0$.

Write $G:=  \prod\nolimits_{j=1}^{j_0} \bigl(GL_k(P_j(n)) \times GL_k(P_j(m)) \bigr)$ and observe that $G$ acts on $R$ as $G=G'\times G''$, where $G'$ is the product over those $j$ up to $j'$, and $G''$ is the product over the remaining indices. The projection $\pi\colon R\to R'$ is $G$-equivariant, where $G'' < G$ acts trivially on $R'$. We now let

$$Z:= \overline{Q^{[\sigma\text{-ss}]}}\subset R\text{ and } Z^{\sigma\text{-ss}} := Z \cap R^{\sigma\text{-ss}},$$ 
and similarly
  $$Z':=\overline{Q'^{[\sigma'\text{-ss}]}} \subset R' \text{ and } Z'^{\,\sigma'\text{-ss}} := Z' \cap R'^{\,\sigma'\text{-ss}}.$$

\subsection{The main technical result}
 By Lemma \ref{lem:exerciseschmitt}, the existence of a good quotient for the $G$-action on $Z^{\sigma\text{-ss}}$ implies the existence of a good quotient $Z^{\sigma\text{-ss}}\hq G''$ for the induced action of $G''$, and our main technical result is to identify this quotient.
\begin{proposition}[Compatibility of the projection with GIT and embedding functors]\label{prop:maintechinical}
The projection $\pi\colon R\to R'$ induces a $G'$-equivariant, $G''$-invariant map $$\pi\colon Z^{\sigma\text{-ss}} \to Z'^{\,\sigma'\text{-ss}}.$$ Moreover, the induced map
$$\hat{\pi} \colon Z^{\sigma\text{-ss}}\hq G'' \to Z'^{\,\sigma'\text{-ss}}$$ is a $G'$-equivariant isomorphism.
\end{proposition}

As a consequence of the previous result,  we can show the moduli space $\mathcal M_{\sigma}$ of $\sigma$-multi-Gieseker semistable sheaves can be obtained as a GIT quotient from the $G$-space $Z$.  This statement will eventually lead to the desired master space construction for any finite number of stability parameters.

\begin{corollary}[Obtaining moduli of multi-Gieseker semistable sheaves from $Z$]\label{cor:obtaining}
With the above notations, we have
$$ \mathcal M_{\sigma} \cong Z^{\sigma\text{-ss}}\hq G.$$
\end{corollary}
\begin{proof}
 Observe first that (semi)stability of a sheaf $E$ with respect to $\sigma$ is precisely the same as (semi)stability with respect to $\sigma'$; i.e., the moduli functors coincide and for the corresponding moduli spaces we have $\mathcal M_{\sigma} = \mathcal M_{\sigma'}$.   The construction of $\mathcal M_{\sigma'}$ described in the  Section \ref{sec:construction}, which applies as $\sigma'$ is positive, says $\mathcal M_{\sigma'} = Z'^{\,\sigma'\text{-ss}}\hq G'$.  Thus,  taking the quotient of $Z^{\sigma\text{-ss}}$ by $G =  G' \times G''$ in stages, by Lemma \ref{lem:exerciseschmitt} and Proposition~\ref{prop:maintechinical} leads to the following sequence of isomorphisms
$$ Z^{\sigma\text{-ss}}\hq G \cong (Z^{\sigma\text{-ss}}\hq G'') \hq G' \cong Z'^{\,\sigma'-ss} \hq G' \cong \mathcal M_{\sigma'}=\mathcal M_{\sigma}$$
and therefore establishes the desired result.
\end{proof}

\section{Proof of the main technical result} 
In order to get a preliminary idea from the GIT-part of the theory as to why the result we are claiming is true, observe that taking the classical invariant-theoretic quotient of $R$ by $G''$ yields $R'$.  Moreover, as the character $\theta_{\sigma'}$ of $G$ computed from $\sigma'$ is really a character of $G'$, i.e., $\theta_{\sigma'}$ is trivial on $G''$, the semistable set $R^{\sigma\text{-ss}}$ is saturated with respect to the good quotient $\pi: R \to R\hq G'' \cong R'$. Hence, the situation is quite clear for the relevant quiver-representations.  What remains is proving that this geometric picture restricts well to the subschemes parametrising representations coming from (semistable regular) sheaves.

The proof of Proposition~\ref{prop:maintechinical} is subdivided into two steps: in Section~\ref{subsect:firsthalf} we prove the first assertion, in Section~\ref{subsect:secondhalf} we establish the second.

\subsection{Proving the first statement of Proposition~\ref{prop:maintechinical}}\label{subsect:firsthalf}
In this section, we investigate the relation between the equivariant geometry of $\pi$ and the categorical embedding provided by Theorem~\ref{thm:categoryembedding}.

We recall from Section~\ref{sec:construction} the definition of the subschemes $Q \subset R$ and $Q' \subset R'$: the first one parametrises those $A$-modules that are in the image of the $\Hom(T, -)$-functor on the category of $(p,\underline{L})$-regular sheaves, the second one those $A'$-modules that are in the image of the $\Hom(T', -)$-functor on the category of $(p,\underline{L}')$-regular sheaves. Inside these there are open subschemes $Q^{[\sigma\text{-ss}]}$ and $Q'^{[\sigma'\text{-ss}]}$ parametrising $\sigma$- and $\sigma'$-semistable sheaves, respectively.  The following result shows that these subschemes are preserved by the projection $\pi\colon R\to R'$.

\begin{lemma}\label{lem:pinducesQ}
The restriction of $\pi\colon R\to R'$ to $Q$ and $Q^{[\sigma\text{-ss}]}$, respectively, induces morphisms 
\begin{equation*}
\pi\colon Q \to Q'\text{ and } \pi\colon Q^{[\sigma\text{-ss}]}  \to  Q'^{[\sigma'\text{-ss}]}.
\end{equation*}
\end{lemma}
\begin{proof}
We start by introducing some terminology.  Given a flat family $\mathscr{M}$ of $A$-modules over $S$ we write $\pi(\mathscr{M})$ to mean the induced family of $A'$-modules over $S$ obtained by ignoring all but the top $j'$ rows of the quiver $\mathcal{Q}$.  If $u\colon S'\to S$ is a morphism then $u^*\mathscr{M}$ is a family of $A$-modules over $S'$ and clearly \begin{equation}\label{eq:commute}
 \pi(u^*\mathscr{M}) = u^*\pi(\mathscr{M}).                                                                                                                                                                                                                                                                                                                                                                                                                                                                                                                                                                                                                                                                                                                  \end{equation}

 Now let $\mathscr{U}$ be the tautological family of $A$-modules parametrised by $R$ and  $\mathscr{U}'$ the tautological family of $A'$-modules parametrised by $R'$. 
Then,
\begin{equation}\pi(\mathscr{U}) = \pi^* \mathscr{U}'\label{eq:compareuniversalfamilies}\end{equation}
as families of $A'$-modules parametrised by $R$.

We first show that the restriction of $\pi$ induces  $\pi\colon R^{[n\text{-reg}]} \to R'^{[n\text{-reg}]}$.     By definition of $\iota\colon R^{[n\text{-reg}]}\to R$ (Proposition \ref{prop:identifyingtheimage} applied with $B,\mathscr{M}$ replaced by $R,\mathscr{U}$) we have that  $\iota^* \mathscr{U} \otimes_{\mathscr A} T$ is an  $R^{[n\text{-reg}]}$-flat family of $(n,\underline{L})$-regular sheaves on $X$ of topological type $\tau$ and the unit map
 \begin{equation*}
 \eta_{\iota^*\mathscr{U}}\colon  \iota^* \mathscr{U} \to \sheafHom_X(T, \iota^* \mathscr{U} \otimes_\mathscr{A} T)\label{eq:unitmap}
 \end{equation*}
is an isomorphism.  Observe that  since each sheaf in the family  $\iota^* \mathscr{U} \otimes_{\mathscr A} T$ is $(n,\underline{L})$-regular, it is also $(n,\underline{L}')$-regular.  Moreover, using \eqref{eq:compareuniversalfamilies} we have
$$\iota^* \pi^* \mathscr{U}' = \iota^* \pi(\mathscr{U}) = \pi(\iota^* \mathscr{U}) \cong \pi(\sheafHom_X(T, \iota^* \mathscr{U} \otimes_\mathscr{A} T))= \sheafHom_X(T', \iota^* \mathscr{U} \otimes_\mathscr{A} T).$$
So, from the defining property of $R'^{[n\text{-reg}]}$ (the second statement in Proposition \ref{prop:identifyingtheimage} with $B,\mathscr{M}$ replaced by $R',\mathscr{U}'$) this precisely says that $\pi\circ\iota\colon R^{[n\text{-reg}]} \to R'$ factors through $R'^{[n\text{-reg}]}$, and moreover
\begin{equation}\label{eq:pullbackuniversal}
\iota^* \mathscr{U} \otimes_\mathscr{A} T \cong  \iota^*\pi^* \mathscr{U}' \otimes_{\mathscr{A'}} T'\end{equation}
as $R^{[n\text{-reg}]}$-flat families of sheaves on $X$ of topological type $\tau$.

Now $Q$ (resp.\ $Q^{[\sigma\text{-ss}]}$) is the open subset of $R^{[n\text{-reg}]}$ consisting of those points over which  $\mathscr{U} \otimes_\mathscr{A} T$ is a sheaf on $X$ that is $(p,\underline{L})$-regular (resp.\ $\sigma$-semistable), and similarly for $Q'$ and $Q'^{[\sigma\text{-ss}]}$.  But a sheaf being being $(p,\underline{L})$-regular (resp.\ $\sigma$-semistable) clearly implies that it is also $(p,\underline{L}')$-regular (resp.\ $\sigma'$-semistable).  Thus, \eqref{eq:pullbackuniversal} implies that $\pi\colon Q\to Q'$ and $\pi\colon Q^{[\sigma\text{-ss}]} \to Q'^{[\sigma\text{-ss}]}$ as claimed.
\end{proof}

Next, we consider the compatibility of $\pi$ with the GIT-stability conditions $\sigma$ and $\sigma'$. This will establish the first statement of Proposition~\ref{prop:maintechinical}.
\begin{lemma}\label{lem:desiredmap}
 As before, let $Z^{\sigma\text{-ss}} = Z \cap R^{\sigma\text{-ss}}$ and  $Z'^{\,\sigma'\text{-ss}} = Z' \cap R'^{\,\sigma'\text{-ss}}$.  Then, the restriction of $\pi$ induces a morphism
  $$\pi\colon Z^{\sigma\text{-ss}}\to Z'^{\,\sigma'\text{-ss}}.$$
\end{lemma}
\begin{proof}
Recall that $Z = \overline{Q^{[\sigma\text{-ss}]}}$ is the scheme-theoretic closure of $Q^{[\sigma\text{-ss}]}$ in $R$ and similarly $Z'= \overline{Q'^{\, [\sigma' \text{-ss}]}}$. We claim that since the restriction of $\pi\colon R\to R'$ to $Q^{[\sigma\text{-ss}]}$ induces $\pi\colon Q^{[\sigma\text{-ss}]} \to Q'^{[\sigma'\text{-ss}]}$, the restriction of $\pi$ to $Z$ induces a morphism $\pi\colon Z\to Z'$ between the corresponding scheme-theoretic closures. To see this, recall that the scheme-theoretic image of a morphism $f\colon X\to Y$ between schemes is the smallest closed subscheme $V$ of $Y$ such that $f$ factors through $V$ \cite[\href{http://stacks.math.columbia.edu/tag/01R6}{Tag 01R6}]{stacks-project}.  Let $f\colon Q^{[\sigma\text{-ss}]}\to R'$ be the composition of the inclusion and the projection $\pi\colon R\to R'$, and let $V$ be the scheme-theoretic closure of $f$.  Then, as $f$ factors through $Q'^{[\sigma'\text{-ss}]}\subset R'$ (by Lemma \ref{lem:pinducesQ}) it also factors through the closed subscheme $\overline{Q'^{[\sigma'\text{-ss}]}}\subset R'$.  Hence, $V\hookrightarrow \overline{Q'^{[\sigma'\text{-ss}]}}= Z'$.  On the other hand, by \cite[\href{http://stacks.math.columbia.edu/tag/01R9}{Tag 01R9}]{stacks-project} the composition $f$ factors through a commutative diagram
$$\begin{xymatrix}
{
Q^{[\sigma\text{-ss}]} \ar[rd] \ar@{^(->}[r] &Z\ar@{^(->}[r] \ar[d] &R\ar[d]^\pi\\
& V\ar@{^(->}[r] &R'.
}
\end{xymatrix}
$$
Thus, the restriction of $\pi$ to $Z$ induces $\pi\colon Z\to Z'$, as claimed.

Next, we claim that $\pi$ maps $R^{\sigma\text{-ss}}$ to $R'^{\,\sigma'\text{-ss}}$, which along with the claim proven in the previous paragraph immediately establishes the statement of the lemma.  But this is clear, for suppose that $M\in R^{\sigma\text{-ss}}$ and $M'$ is a proper $A'$-submodule of $\pi(M)$.   Then, we may extend $M'$ to an $A$-module $M_0$ by associating the zero vector space for all vertices of $\mathcal{Q}$ that are not in the first $j'$ rows and also associating the zero morphism to any arrow that starts or ends outside the first $j'$ rows.  Clearly, $M_0$ is a proper $A$-submodule of $M$, and hence
$$\theta_{\sigma'}(M') = \theta_{\sigma}(M_0) \le 0,$$
where the inequality uses semistability of $M$.  Thus, $\pi(M)\in R'^{\,\sigma'\text{-ss}}$,  as required.
\end{proof}

Now, as $Z^{\sigma\text{-ss}}$ admits a good quotient by the $G$-action, by Lemma~\ref{lem:exerciseschmitt} it also admits a good quotient by $G''$.     As good quotients are categorical, the $G''$-invariance of $\pi$ together with Lemma~\ref{lem:desiredmap} yields an induced morphism
\[\hat{\pi}\colon Z^{\sigma\text{-ss}}\hq G'' \to Z'^{\,\sigma'\text{-ss}}.\]
By construction, $\hat\pi$ is $G'$-equivariant.

\subsection{Proving the second statement of Proposition~\ref{prop:maintechinical}}\label{subsect:secondhalf}
In order to prove the second statement of Proposition~\ref{prop:maintechinical}, we now construct an inverse to the morphism $\hat{\pi}$.  First observe that since $\sigma'$ is a positive stability parameter, Theorem \ref{thm:moduliexist} gives
$$ Z'^{\,\sigma'\text{-ss}} = Q'^{[\sigma'\text{-ss}]}.$$
So, our aim is the construction of a natural morphism
$$ \hat{s} \colon  Q'^{[\sigma'\text{-ss}]} \to Z^{\sigma\text{-ss}}\hq G''.$$
 To this end, as before let $\mathscr{U}'$ be the tautological family of $A'$-modules parametrised by $R'$ and let $\iota\colon  Q'^{[\sigma'\text{-ss}]} \to R'$ be the inclusion.  So, by construction
\begin{equation}\mathcal{E}': =\iota^*\mathscr{U}' \otimes_{\mathscr{A}'} T'\label{eq:mathscrE'}\end{equation}
is a $Q'^{[\sigma'\text{-ss}]}$-flat family of sheaves on $X$ of topological type $\tau$ that are $\sigma'$-semistable, and the unit map
 \begin{equation}\eta\colon  \iota^* \mathscr{U}' \to \sheafHom_X(T', \mathcal{E}')\label{eq:unitmapiso}\end{equation}
is an isomorphism.

Consider now the family of $A$-modules
\begin{equation}\mathscr{N}:= \sheafHom_X(T,\mathcal{E}')\label{eq:defscrN}\end{equation}
over $Q'^{[\sigma'\text{-ss}]}$.  Since each sheaf in the family $\mathcal E'$ is $\sigma'$-semistable, it is also $\sigma$-semistable and hence $(p,\underline{L})$-regular by our choice of integers $m\gg n\gg p$.  Thus, $\mathscr{N}$ is a flat family of $A$-modules of dimension $\underline{d}$ (flatness of $\mathscr{N}$ follows from flatness of $\mathcal{E}'$ and $(n,\underline{L})$-regularity as explained in the proof of \cite[Proposition 5.8]{GRTI}). Moreover, by preservation of semistability, Theorem~\ref{thm:mainsemistabilitycomparison}(1), $\mathscr{N}$ is in fact a flat family of $\sigma$-semistable $A$-modules.

We now define our desired map $\hat{s}$, first over a sufficiently small affine open subset $\Omega \subset Q'^{[\sigma'\text{-ss}]}$.  To do this, recall that $\mathscr{U}'$ is canonically trivialised, and hence by \eqref{eq:unitmapiso} we have a trivialisation of $\sheafHom_X(T',\mathcal{E}')$.  Said another way, the vector bundles in $\mathscr{N}$ that are associated with the vertices of the subquiver $\mathcal{Q}'$ are trivialised.   If $\Omega$ is sufficiently small, we may pick a trivialisation of the remaining part of $\mathscr{N}$ (so trivialisations for the vector bundles associated with the vertices outside of $\mathcal{Q}'$).  Denote this trivialisation  by $\phi$.  Then, $\mathscr{N}$ is trivialised over $\Omega$, and so induces a morphism $s_{\phi}\colon \Omega \to R$ fulfilling 
\begin{equation}\label{eq:pullbackuniversal1}
s_{\phi}^*\mathscr{U} = \mathscr{N}
\end{equation}
(see the discussion at the end of Section \ref{sec:families}). From the defining properties of $Q^{[\sigma\text{-ss}]}$ and $R^{\sigma\text{-ss}}$ we infer that $s_\phi$ factors through a morphism $s_\phi\colon \Omega \to   \overline{Q^{[\sigma\text{-ss}]}}  \cap R^{\sigma\text{-ss}} = Z^{\sigma\text{-ss}}$.  We let 
$$\hat{s}_\phi \colon \Omega \to   Z^{\sigma\text{-ss}}\hq G''$$
be the composition of $s_\phi$ with the good quotient $Z^{\sigma\text{-ss}} \to Z^{\sigma\text{-ss}}\hq G''$.  Observe that, essentially by construction, $\hat{s}_{\phi}$ is independent of the chosen trivialisation $\phi$, since any other choice can be obtained by composing $\phi$ with the action of a morphism $\Omega\to G''$, and since the quotient map $Z^{\sigma\text{-ss}} \to Z^{\sigma\text{-ss}}\hq G''$ is constant on $G''$-orbits.  Thus, we may glue the locally defined $s_\phi$ to obtain the desired global morphism
$$ \hat{s} \colon  Q'^{[\sigma'\text{-ss}]} \to Z^{\sigma\text{-ss}}\hq G''.$$

Now, recall that from Lemma \ref{lem:desiredmap} and the subsequent discussion that we have already obtained morphisms fitting into the following diagram
$$\begin{xymatrix}
{
Z^{\sigma\text{-ss}} \ar[r]^{\pi}\ar@{->>}[d] &Z'^{\,\sigma'\text{-ss}}\\
Z^{\sigma\text{-ss}} \hq G''. \ar[ru]_{\hat{\pi}}
}
\end{xymatrix}
$$
In the final step, we prove compatibility of $\hat \pi$ with the morphism $\hat s$ just constructed.
\begin{lemma}
There is a  commutative diagram
$$\begin{xymatrix}
{ Q'^{[\sigma'\text{-ss}]} \ar[r]^{\hat{s}} \ar[rd]_{\hat{\pi}\circ\hat{s}} & Z^{\sigma\text{-ss}}\hq G'' \ar[d]^{\hat{\pi}} & \\
 & \overline{ \pi(Z^{\sigma\text{-ss}} )} \ar@{^(->}[r]^-{\iota} &  Z'^{\,\sigma'\text{-ss}}= Q'^{[\sigma'\text{-ss}]}, }
\end{xymatrix}$$
where $\overline{\pi(Z^{\sigma\text{-ss}})}$ denotes scheme-theoretic image of $\pi$, and $\iota$ is the natural inclusion. Moreover, we have
$$\iota \circ \hat{\pi}\circ \hat{s}= \id \colon Q'^{[\sigma'\text{-ss}]} \to Q'^{[\sigma'\text{-ss}]} .$$
\end{lemma}
\begin{proof}

That $\hat{\pi}$ factors this way is clear from the definition. It is also clear from the above construction that $\iota\circ\hat{\pi}\circ \hat{s}$ is the identity pointwise.  This also holds in families, as using the above discussion (in particular \eqref{eq:commute}, \eqref{eq:compareuniversalfamilies}, \eqref{eq:mathscrE'}, \eqref{eq:unitmapiso}, \eqref{eq:defscrN}, \eqref{eq:pullbackuniversal1}), over a sufficiently small affine subset $\Omega$ of $Q'^{[\sigma'\text{-ss}]}$ allowing for a trivialisation $\phi$ of $\mathscr{N}$ as above, we have canonical identifications
$$ (\pi s_{\phi})^*  \mathscr{U}'\negthinspace = \negthinspace s_{\phi}^* \pi^* \mathscr{U}' \negthinspace = \negthinspace s_{\phi}^* \pi(\mathscr{U}) \negthinspace = \negthinspace \pi( s_{\phi}^* \mathscr{U})\negthinspace = \negthinspace \pi(\mathscr{N}) \negthinspace = \negthinspace \pi(\sheafHom_X(T,\mathcal{E}')) \negthinspace = \negthinspace \sheafHom_X(T',\mathcal{E}') \negthinspace =  \negthinspace \mathscr{U}'\negthinspace,  $$
which shows that $\pi\circ s_{\phi} = \id$ on $\Omega$.  Hence, $\hat{\pi}\circ \hat{s}$ is the identity of $Q'^{[\sigma'\text{-ss}]}$, as claimed.
\end{proof}

\begin{lemma}
 The scheme-theoretic image $\overline{\pi(Z^{\sigma\text{-ss}})}$ coincides with $ Q'^{[\sigma'\text{-ss}]} =Z'^{\,\sigma'\text{-ss}}$.
\end{lemma}
\begin{proof}
 Consider the sequence of maps
  $$  Q'^{[\sigma'\text{-ss}]}= Z'^{\,\sigma'\text{-ss}}\xrightarrow{\; \hat{\pi}\circ \hat{s}\;} \overline{\pi(Z^{\sigma\text{-ss}})} \stackrel{\iota}{\hookrightarrow} Z'^{\,\sigma'\text{-ss}}.$$
 The fact that $\iota \circ \hat{\pi}\circ \hat{s}=\id$ in a first step gives equality $\overline{\pi(Z^{\sigma\text{-ss}})}_{red} = ( Z'^{\,\sigma'\text{-ss}})_{red}$ on the level of reduced spaces. But this implies in a second step that the scheme structures also agree: indeed, since these schemes have the same support we may work affine locally, so the scheme structures are related by ring morphisms  $A\stackrel{\iota^*}{\to} B \to A$ that compose to the identity. As at the same time $\iota^*$ is surjective, it has to be an isomorphism.
\end{proof}

Consequently, the map
$$ \hat{\pi}\colon Z^{\sigma\text{-ss}}\hq G'' \to Z'^{\,\sigma'\text{-ss}}$$
is a $G'$-equivariant isomorphism, establishing the second assertion of Proposition~\ref{prop:maintechinical}.

\section{Master space construction}\label{sect:master_construction}

We are now ready to prove our main variation result.  We continue to use the notation introduced above, so $L_1,\ldots, L_{j_0}$ are fixed very ample line bundles on $X$ and we have fixed a topological type $\tau$.  Suppose $\Sigma \subset \mathbb \mathbb Q^{j_0}_{\ge 0} \{ 0 \}$ is a finite collection of rational stability parameters, and we assume that the family consisting of those sheaves of topological type $\tau$ that are semistable with respect to some $\sigma\in \Sigma$ is bounded.  We choose $m\gg n\gg p$ so the conclusions of Theorem \ref{thm:categoryembedding} and Theorem \ref{thm:mainsemistabilitycomparison}(1)  hold for each $\sigma \in \Sigma$.  By enlarging these integers if necessary, we may assume Theorems \ref{thm:categoryembedding} and \ref{thm:mainsemistabilitycomparison} hold also for the truncation of each $\sigma^{(i)}$ in which all zero entries are ignored; cf.~the organisation of our setup in Section~\ref{sect:setup}. 

Let $\mathcal{Q}$ and $\underline{d}$ be as before, and consider $R =\mathrm{Rep}(\mathcal{Q},\underline{d})$ with the group action of $$G= \prod\nolimits_{j=1}^{j_0} \bigl(GL_k(P_j(n)) \times GL_k(P_j(m)) \bigr).$$

\begin{definition}[The master space]
Define
$$ Y := \bigcup\nolimits_{\sigma \in \Sigma} \overline{Q^{[\sigma\text{-ss}]}}\subset R, $$
where  $\overline{Q^{[\sigma\text{-ss}]}}$ denotes the scheme-theoretic closure of $Q^{[\sigma\text{-ss}]}$ in $R$.
\end{definition}
The following is our main result.
\begin{theorem}[A master space for moduli of multi-Gieseker semistable sheaves]\label{thm:master}
The affine scheme $Y$ is a master space for the moduli spaces $\mathcal M_{\sigma}$ as $\sigma$ varies in $\Sigma$.  More precisely, for each $\sigma \in \Sigma$, we have
$$ \mathcal M_{\sigma} = Y \hq _{\theta_\sigma} G := Y^{\sigma\text{-ss}} \hq G.$$
In particular, any two such moduli spaces are related by a finite number of Thaddeus-flips.
\end{theorem}
\begin{proof} 
Fix $\sigma\in \Sigma$.  Given the work from the previous section, the only issue is in gaining control over those components of $Y$ that do not lie in $\overline{Q^{[\sigma\text{-ss}]}}$.  To this end, let $Z := \overline{Q^{[\sigma\text{-ss}}]}$.
We claim that
\begin{equation}
Y^{\sigma\text{-ss}} = Z^{\sigma\text{-ss}}.\label{eq:masterclaim}\end{equation}
Clearly, we have a scheme-theoretic inclusion $Z^{\sigma\text{-ss}}\subset Y^{\sigma\text{-ss}}$, since $Z\subset Y$ is a closed embedding.

To prove the other inclusion, it is enough to check that for any $\sigma'\in\Sigma$ we have a scheme-theoretic inclusion 
\begin{equation}\label{eq:incl1}
 \overline{Q^{[\sigma'\text{-ss}]}}\cap R^{\sigma\text{-ss}} \subset Z^{\sigma\text{-ss}}.
\end{equation}
Observe first that  by our choice of integers $m$, $n$, and $p$, Theorem \ref{thm:mainsemistabilitycomparison}(1) implies $Q^{[\sigma\text{-ss}]} \subset R^{\sigma\text{-ss}}$.
We claim that \eqref{eq:incl1} is implied by the inclusion
\begin{equation}
Q^{[\sigma'\text{-ss}]}\cap R^{\sigma\text{-ss}}\subset Q^{[\sigma\text{-ss}]}.\label{eq:secondmasterclaim}\end{equation}
To see this, observe that \eqref{eq:incl1} is not affected by any component of $Q^{[\sigma'\text{-ss}]}$ that does not meet the Zariski open set $R^{\sigma\text{-ss}}$, and thus we may assume that $Q^{[\sigma'\text{-ss}]}\cap R^{\sigma\text{-ss}}$ is Zariski dense in $Q^{[\sigma'\text{-ss}]}$, at which point  \eqref{eq:incl1}  follows formally from \eqref{eq:secondmasterclaim}.

Now, the definition of $Q^{[\sigma'\text{-ss}]}$ and Proposition \ref{prop:identifyingtheimage}(a) show that $Q^{[\sigma'\text{-ss}]}$ is the base of a flat family of $\sigma'$-semistable $(p, \underline{L})$-regular sheaves of topological type $\tau$ on $X$. Each sheaf in this family that lies over a point of the intersection $Q^{[\sigma'\text{-ss}]} \cap R^{\sigma\text{-ss}}$ corresponds to a $\sigma$-semistable $A$-module and thus by Theorem \ref{thm:mainsemistabilitycomparison}(1) is $\sigma$-semistable as a sheaf on $X$. Therefore, Proposition \ref{prop:identifyingtheimage}(b) and the definition of $Q^{[\sigma\text{-ss}]}$ imply that the inclusion $Q^{[\sigma'\text{-ss}]} \cap R^{\sigma\text{-ss}}\to R^{\sigma\text{-ss}}$ factors through  $Q^{[\sigma\text{-ss}]}$, which proves \eqref{eq:secondmasterclaim} and as consequence also \eqref{eq:incl1} and	 \eqref{eq:masterclaim}.

Finally, from \eqref{eq:masterclaim} and Corollary \ref{cor:obtaining} we obtain 
$$ Y \hq_{\theta_{\sigma}} G \cong Y^{\sigma\text{-ss}}\hq 	G \cong Z^{\sigma\text{-ss}}\hq G \cong \mathcal M_{\sigma},$$
which concludes the proof of Theorem~\ref{thm:master}.
\end{proof}

\begin{remark}[Sheaf-semistability vs.~quiver-semistability]
 Note that in contrast to \cite[Thm.~10.1]{GRTI} we do not claim that $Y^{\sigma\text{-ss}} = Q^{[\sigma\text{-ss}]}$. As the stability parameter $\sigma$ is degenerate, the embedding of $Q^{[\sigma\text{-ss}]}$ into $R^{\sigma\text{-ss}}$ is not saturated with respect to the GIT-quotient $R^{\sigma\text{-ss}} \to R^{\sigma\text{-ss}}\hq G$. This is intimately related to the fact that one needs the relevant stability parameter to be positive for Theorem~\ref{thm:mainsemistabilitycomparison}(2) to hold. 
\end{remark}

As discussed in the Introduction,  as a consequence of Theorem~\ref{thm:master} we obtain the desired VGIT statement about the relation between Gieseker-moduli spaces of semistable sheaves with respect to two choices of ample line bundles on a fixed base scheme $X$.

\begin{corollary}[Mumford-Thaddeus principle for Gieseker-moduli spaces]\label{cor:master}
Let $L_0$ and $L_1$ be ample line bundles on $X$.  Then, the moduli spaces $\mathcal M_{L_0}$ and $\mathcal M_{L_1}$ of sheaves of a given topological type that are Gieseker-semistable with respect to $L_0$ and $L_1$, respectively, are related by a finite number of Thaddeus-flips.
\end{corollary}
\begin{proof}
Without loss of generality we may assume $L_0$ and $L_1$ are very ample.  Then, apply the previous theorem to the stability parameters $(L_0,L_1,1,0)$ and $(L_0,L_1,0,1)$ for the quiver $\mathcal{Q}$ we obtain setting $j_0 = 2$. Note that the required boundedness requirements are fulfilled due to classical results about Gieseker-semistability, see for example \cite[Theorem~3.3.7]{Bible}. 
\end{proof}

\begin{remark}[VGIT for a finite number of moduli spaces]
 It is obvious from the above that we can discuss any number $j_0$ of ample polarisations and their induced Gieseker moduli spaces simultaneously using the quiver $\mathcal{Q}$ with the appropriate number $j_0$ of rows. 
\end{remark}

\section{Intermediate Spaces and Uniformity}\label{sec:uniform}

The master space construction above works for any finite collection of stability parameters.  However it has a deficiency in that it does not identify the intermediate spaces that appear in the Thaddeus-flips occuring between different moduli spaces.  This is discussed in detail in our previous work \cite{GRTIIpreprint}, the upshot being that given a family $(\sigma(t))_{t\in [0,1]}$ of stability parameters (all taken with respect to the same vector $\underline{L}$ of ample line bundles), one can only reasonably expect to be able to control these intermediate spaces if one makes some additional assumptions that we discuss next. For this section we assume that $X$ is smooth of dimension $d$ and all the sheaves in question are torsion-free.

Fix ample line bundles $L_j$ for $j=1,\ldots,j_0$ on X.  We extend the notion of multi-Gieseker stability to allow  twisting by a fixed collection of line bundles $B_j$, $j=1,\ldots,j_0$. Given these data, the \emph{multi-Hilbert polynomial}  with respect to $\sigma\in (\mathbb R_{\ge 0})^{j_0}\setminus\{0\}$ of a coherent sheaf $E$ is now taken to be 
   \begin{equation}\label{eq:defalpha:twisted}
     P_E^\sigma(m) := \sum\nolimits_{j} \sigma_j \chi(E\otimes L_j^m \otimes B_j).
   \end{equation}
This twisting adds no new difficulties (see \cite{GRTIIpreprint}).  We say a path $\sigma\colon [0,1] \to \mathbb R_{\ge 0}^{j_0}\setminus \{0\}$,  $\sigma(t) = (\sigma_1(t),\ldots,\sigma_{j_0}(t))$ is a \emph{stability segment} if each $\sigma_j$ is a linear function of $t$ and
 $$\sum\nolimits_j \vol(L_j) \sigma_j(t) = 1 \quad \text{ for all } t \in [0,1],$$
 where $\vol(L_j) : = \int_X c_1(L_j)^d$.  We say a stability
segment  $(\sigma(t))_{t\in [0,1]}$ is \emph{bounded} if the set of sheaves of a given topological type that are
semistable with respect to $\sigma(t)$ for some $t\in [0,1]$ is bounded.
 
\begin{definition}[Uniform stability]
We say a stability segment $(\sigma(t))_{t\in [0,1]}$  is \emph{uniform} if for every torsion-free sheaf $E$ and every $t\in [0,1]$ we have
$$  \frac{\sum_{j}(t) \sigma_j \chi(E\otimes L_{j}^{k}\otimes B_{j})}{\rank(E)}=  \frac{k^d}{d!} + a_{d-1}(E) k^{d-1} + \cdots + a_1(E)k + a_0(E,t),$$
where  $a_{d-1}(E),\ldots,a_{1}(E)$ are independent of $t$ and $a_0(E,t)$ is linear in $t$.
\end{definition}

Key to the notion of uniform stability segment is the following \emph{semicontinuity property}:  if $E$ is a sheaf that is semistable with respect to $\sigma(t)$ for all $t<\overline{t}$ then it is also semistable with respect to $\sigma(\overline{t})$ (this is easy to see, and discussed further in \cite[Remark 2.5]{GRTIIpreprint}).   

\begin{theorem}[Thaddeus-flips through moduli spaces of sheaves] \label{thm:flipsthoughmoduli}
Let $X$ be smooth and projective, let $\tau$ be a topological type and $(\sigma(t))_{t\in [0,1]}$ be a bounded uniform stability segment.   For $t\in [0,1]$ let $\mathcal M_{\sigma(t)}$ be the moduli space of torsion-free sheaves on $X$ of topological type $\tau$ that are $\sigma(t)$-semistable. 

Then given any $t',t''$ in $[0,1]$ the moduli spaces $\mathcal M_{\sigma(t')}$ and $\mathcal M_{\sigma(t'')}$ are connected by a finite collection of Thaddeus-flips of the form
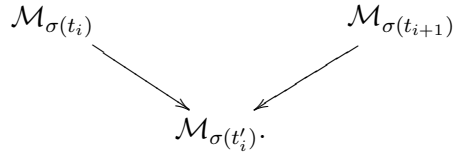
\[\begin{xymatrix}{
 \mathcal{M}_{\sigma(t_i)} \ar[rd] &  & \mathcal{M}_{\sigma(t_{i+1})} \ar[ld] \\
                 &        \mathcal{M}_{\sigma(t'_i)}.               & }
\end{xymatrix}
\]
for some $t_i,t_i'\in [0,1]$. 
\end{theorem}
\begin{proof}
We remark that the only difference between the statements of Theorem~\ref{thm:flipsthoughmoduli} and \cite[Theorem 2.6]{GRTIIpreprint} is that here we allow the parameters $t'$ and $t''$ to be among the endpoints $t=0$ or $t=1$.  The crucial point is that the uniformity assumption on $(\sigma(t))_{t\in [0,1]}$ allows us to choose the integers $m\gg n\gg p$ appearing in the proof of Theorem \ref{thm:master} in such a way that they work uniformly over all $t\in [0,1]$.  The rest of the proof is then essentially the same as either Theorem \ref{thm:master} or \cite[Theorem 2.6]{GRTIIpreprint}, and so we do not repeat it here.  We merely observe that the newly established results of Section \ref{sec:degenerate} allow us to use the same master space up to the endpoints, although $\sigma(0)$ and $\sigma(1)$ may not be positive stability parameters.
\end{proof}

\vspace{1cm}


\begin{thebibliography}{MFK94}
\addtocontents{toc}{\protect\setcounter{tocdepth}{1}}



\bibitem[ACK07]{ConsulKing}  Luis \'Alvarez-C\'onsul and Alastair King, \emph{A functorial construction of moduli of sheaves}, Invent. Math. \textbf{168} (2007), no.~3, 613--666.

\bibitem[BM15]{BertramMartinez} Aaron Bertram and Cristian Martinez, \emph{Change of polarization for moduli of sheaves on surfaces as Bridgeland wall-crossing}, preprint \texttt{arXiv:1505.07091}, 2015.


\bibitem[DH98]{DolgachevHu} Igor Dolgachev and Yi Hu, \emph{Variation of geometric invariant theory quotients}, Inst. Hautes \'Etudes Sci. Publ. Math. \textbf{87} (1998), 5--56. 

\bibitem[Ful98]{F} William Fulton,  \emph{Intersection theory}, Springer-Verlag, Berlin, 1998.

\bibitem[GRT14]{GRTI}
Daniel Greb, Julius Ross, and Matei Toma,\emph{Variation of Gieseker moduli spaces via quiver GIT}, to appear in Geometry $\&$ Topology, preprint \texttt{arXiv:1409.7564}, 2014.

\bibitem[GRT15]{GRTIIpreprint} Daniel Greb, Julius Ross, and Matei Toma, \emph{Semi-continuity of stability for sheaves and variation of Gieseker moduli spaces}, to appear in Journal f\"ur die Reine und Angewandte Mathematik (Crelle), preprint \texttt{arXiv:1501.04440}, 2015.



 \bibitem[HL10]{Bible}  Daniel Huybrechts and Manfred Lehn, \emph{The geometry of moduli spaces of
   sheaves}, second ed., Cambridge Mathematical Library, Cambridge University Press, Cambridge, 2010.


\bibitem[Kin94]{King} Alastair King, \emph{Moduli of representations of finite-dimensional algebras}, Quart. J. Math.~(2) \textbf{45} (1994),  515--530.

\bibitem[MFK94]{MumfordGIT}
David Mumford, John Fogarty, and Frances Kirwan, \emph{Geometric {I}nvariant {T}heory}, Ergebnisse der Mathematik und ihrer Grenzgebiete (2), vol.~34, Springer-Verlag, Berlin, 1994.


\bibitem[MW97]{MatsukiWentworth}
Kenji Matsuki and Richard Wentworth, \emph{Mumford-{T}haddeus principle on the
  moduli space of vector bundles on an algebraic surface}, Internat. J. Math.
  \textbf{8} (1997), no.~1, 97--148.
  
\bibitem[Ram96]{Ramanathan} A. Ramanathan, \emph{Moduli for principal bundles over algebraic curves.~II}, Proc. Indian Acad. Sci. Math. Sci. \textbf{106} (1996), no. 4, 421--449.
  
\bibitem[Sch00]{Schmitt}
Alexander Schmitt, \emph{Walls for {G}ieseker semistability and the {M}umford-{T}haddeus principle for moduli spaces of sheaves over higher
  dimensional bases}, Comment. Math. Helv. \textbf{75} (2000), no.~2, 216--231.
  
\bibitem[Sch08]{SchmittBook} 
Alexander Schmitt, \emph{Geometric invariant theory and decorated principal bundles},  Zurich Lectures in Advanced Mathematics, European Math. Soc. Publishing House, Z\"urich, 2008. 
  
\bibitem[Sim94]{Simpson} Carlos Simpson, \emph{Moduli of representations of the fundamental group of a smooth projective variety.~I}, Inst. Hautes \'Etudes Sci. Publ. Math. \textbf{79} (1994), 47--129.

\bibitem[SP16]{stacks-project} The Stacks Projects Authors, \emph{Stacks Project}, \url{http://stacks.math.columbia.edu}, 2016.


\bibitem[Tha96]{Thaddeus} Michael Thaddeus, \emph{Geometric invariant theory and flips},
J. Amer. Math. Soc. \textbf{9} (1996), no. 3, 691--723. 

\end{thebibliography}
\end{document}